\documentclass[a4paper]{amsart}

\usepackage{mathtools}

\makeatletter
\@namedef{subjclassname@2020}{\textup{2020} Mathematics Subject Classification}
\makeatother

\theoremstyle{plain}
\newtheorem{theorem}{Theorem}[section]
\newtheorem{corollary}[theorem]{Corollary}
\newtheorem{proposition}[theorem]{Proposition}
\newtheorem{lemma}[theorem]{Lemma}
\newtheorem{claim}[theorem]{Claim}
\newtheorem{question}[theorem]{Question}

\theoremstyle{remark}
\newtheorem{remark}[theorem]{Remark}

\begin{document}

\title[A universality property]{Orbit pseudometrics and a universality property of the Gromov-Hausdorff distance}
\author{Ond\v{r}ej Kurka}
\thanks{The author was supported by the GA\v{C}R project 20-22230L and RVO: 67985840.}
\address{Institute of Mathematics, Czech Academy of Sciences, \v{Z}itn\'a 25, 115 67 Prague~1, Czech Republic}
\email{kurka.ondrej@seznam.cz}
\begin{abstract}
We consider the notion of Borel reducibility between pseudometrics on standard Borel spaces introduced and studied recently by C\'{u}th, Doucha and Kurka, as well as the notion of an orbit pseudometric, a continuous version of the notion of an orbit equivalence relation. It is well known that the relation of isometry of Polish metric spaces is bireducible with a universal orbit equivalence relation. We prove a version of this result for pseudometrics, showing that the Gromov-Hausdorff distance of Polish metric spaces is bireducible with a universal element in a certain class of orbit pseudometrics.
\end{abstract}
\maketitle

\section{Introduction and main results}

Let $ E, F $ be equivalence relations on Polish spaces $ X, Y $ (see Section~\ref{section:prelim} for more details on the terminology). We say that $ E $ is \emph{Borel reducible} to $ F $ if there exists a Borel mapping $ f : X \to Y $ (so-called \emph{reduction}) such that
$$ f(x) \, F \, f(x') \quad \Leftrightarrow \quad x \, E \, x' $$
for all $ x, x' \in X $. In some sense, this means that $ E $ is at most as complex as $ F $, or that the problem of deciding whether two objects are $ E $-equivalent can be reduced to the problem of deciding whether two objects are $ F $-equivalent.

The study of reducibility between definable equivalence relations, often called invariant descriptive set theory, is an active area of current research, see e.g.~\cite{gao}. Let us emphasize here the result of Clemens, Gao and Kechris \cite{clemens, gaokech} which states that the relation of isometry of Polish metric spaces is bireducible with a universal orbit equivalence relation. Here, an \emph{orbit equivalence relation} is the relation defined by $ x E^{X}_{G} y \Leftrightarrow \exists g \in G : g \cdot x = y $ for some Borel action of a Polish group $ G $ on a Polish space $ X $, and an orbit equivalence relation $ E^{X}_{G} $ is called \emph{universal} if any other orbit equivalence relation $ E^{Y}_{H} $ is Borel reducible to $ E^{X}_{G} $.

It is worth noting that the isometry of Polish metric spaces is Borel bireducible also with the linear isometry of separable Banach spaces \cite{melleraylinisom}, with the isomorphism of separable C*-algebras \cite{sabok} and with the homeomorphism of compact metric spaces \cite{zielinski}.

Recently, the notion of Borel reducibility has been investigated by C\'{u}th, Doucha and Kurka in the setting of pseudometrics on Polish spaces, see \cite{cdk1, cdk2}. This is a more general setting, since an equivalence relation $ E $ can be viewed as the pseudometric $ \varrho_{E} $ defined by $ \varrho_{E}(x, y) = 0 $ if $ xEy $ and $ \varrho_{E}(x, y) = 1 $ otherwise. Moreover, in \cite{cdk1}, it is shown that orbit equivalence relations have a natural generalization in the setting of pseudometrics, so-called orbit pseudometrics.

Let $ \varrho_{X} $ and $ \varrho_{Y} $ be pseudometrics on Polish spaces $ X $ and $ Y $. We say that $ \varrho_{X} $ is \emph{Borel-uniformly continuous (Borel-u.c.) reducible} to $ \varrho_{Y} $ if there exists a Borel mapping $ f : X \rightarrow Y $ such that, for every $ \varepsilon > 0 $, there are $ \delta_{X} > 0 $ and $ \delta_{Y} > 0 $ satisfying
$$ \varrho_{X}(x, x') < \delta_{X} \; \Rightarrow \; \varrho_{Y}(f(x), f(x')) < \varepsilon $$
and
$$ \varrho_{Y}(f(x), f(x')) < \delta_{Y} \; \Rightarrow \; \varrho_{X}(x, x') < \varepsilon $$
for all $ x, x' \in X $. This is a quantitative version of the Borel reducibility of the relation $ \varrho_{X}(x, x') = 0 $ to the relation $ \varrho_{Y}(y, y') = 0 $. Let us accentuate that pseudometrics in this work are allowed to attain $ \infty $ as a value.

The central object studied in \cite{cdk1, cdk2} is the Gromov-Hausdorff distance of Polish metric spaces, which, in some sense, measures how far are two spaces from being isometric (see Section~\ref{section:prelim} for the precise definition, as well as for the definitions of the Urysohn space $ \mathbb{U} $ and of the coding $ F(\mathbb{U}) $ of Polish metric spaces, enabling us to see the Gromov-Hausdorff distance as a pseudometric $ \varrho_{GH} $ on a Polish space). It is shown that the Gromov-Hausdorff distance is bireducible with several other distances, for instance the Lipschitz distance of Polish metric spaces or the Kadets and Banach-Mazur distances of separable Banach spaces. Also, it is bireducible with a simple looking orbit pseudometric, we recall this in Theorem~\ref{thmcdk}.

We now formulate the main result of this work. This is a continuous version of the result of Clemens, Gao and Kechris, in which the Gromov-Hausdorff distance takes over the role of the isometry relation.

\begin{theorem} \label{thmmain1}
Let $ G $ be a Polish group acting continuously on a Polish space $ X $. Let $ d $ be a lower semicontinuous pseudometric on $ X $ such that $ d(x, y) = d(gx, gy) $ for any $ x, y \in X $ and $ g \in G $. Moreover, let there be a system $ \Phi $ of continuous pseudometrics on $ X $ such that $ d = \sup_{\varphi \in \Phi} \varphi $.

Let $ \varrho_{G, d} $ be the corresponding orbit pseudometric, i.e., the pseudometric defined by
$$ \varrho_{G, d}(x, y) = \inf_{g \in G} d(gx, y). $$
Then $ \varrho_{G, d} $ is Borel-u.c.~reducible to the Gromov-Hausdorff distance $ \varrho_{GH} $.
\end{theorem}

It is not clear if the theorem holds without the assumption of the existence of a system $ \Phi $ of continuous pseudometrics (this is extensively discussed in Section~\ref{section:comments}). Actually, we prove a version of the theorem with a more general system $ \Phi $ and a stronger conclusion, see Theorem~\ref{thmmain3}. The constructed reduction $ f : X \to F(\mathbb{U}) \setminus \{ \emptyset \} $ satisfies 
$$ \varrho_{GH}(f(p), f(q)) \leq \min \{ \varrho_{G, d}(p, q), 1 \} \leq 2 \varrho_{GH}(f(p), f(q)), \quad p, q \in X. $$
Thus, in the terminology of \cite{cdk1}, the reduction is Borel-Lipschitz on small distances.

Let us recall a remarkable problem posed by Ben Yaacov, Doucha, Nies and Tsankov \cite{bydnt}.
Let $ E_{GH} $ be the relation given by $ M E_{GH} N \Leftrightarrow \varrho_{GH}(M, N) = 0 $ for $ M, N \in F(\mathbb{U}) \setminus \{ \emptyset \} $. Then any orbit equivalence is Borel reducible to $ E_{GH} $, but it is not known if $ E_{GH} $ is Borel reducible to an orbit equivalence, see \cite[Question~8.6]{bydnt}. However, $ E_{GH} $ possesses some properties of orbit equivalences, as its equivalence classes are Borel, see \cite[Corollary~8.3]{bydnt} (or \cite[Corollary~52]{cdk1}), and the relation
$$ x E_{1} y \; \Leftrightarrow \; \big( \exists n \, \forall m \geq n : x(m) = y(m) \big), \quad x, y \in (2^\mathbb{N})^{\mathbb{N}}, $$
is not Borel reducible to it, see \cite[Theorem~24]{cdk1}. Therefore, Theorem~\ref{thmmain1} has the following consequence.

\begin{corollary} \label{corzerodist}
Let $ \varrho_{G, d} $ be as in Theorem~\ref{thmmain1}, and let $ E^{X}_{G, d} $ be the equivalence relation given by $ x E^{X}_{G, d} y \Leftrightarrow \varrho_{G, d}(x, y) = 0 $. Then $ E^{X}_{G, d} $ is Borel reducible to $ E_{GH} $, and thus
\begin{itemize}
\item the equivalence classes of $ E^{X}_{G, d} $ are Borel,
\item $ E_{1} $ is not Borel reducible to $ E^{X}_{G, d} $.
\end{itemize}
\end{corollary}

It turns out that our methods moreover provide a positive answer to Question~41 from \cite{cdk1}. Let us recall that Gao and Kechris \cite{gaokech} proved that the relation induced by the canonical action of the isometry group $ Iso(\mathbb{U}) $ on $ F(\mathbb{U}) $ is a universal orbit equivalence relation (and so it is Borel bireducible with the isometry relation of Polish metric spaces). Considering the Hausdorff distance $ \varrho_{H} $, we can obtain the following analogue for pseudometrics.

\begin{theorem} \label{thmmain2}
The Gromov-Hausdorff distance $ \varrho_{GH} $ is Borel-u.c.~bireducible with the orbit pseudometric $ \varrho_{Iso(\mathbb{U}), \varrho_{H}} $ on $ F(\mathbb{U}) \setminus \{ \emptyset \} $.
\end{theorem}

This theorem is also a consequence of Theorem~\ref{thmmain3} which provides a reduction of an orbit pseudometric to both $ \varrho_{GH} $ and $ \varrho_{Iso(\mathbb{U}), \varrho_{H}} $. Similarly as in Theorem~\ref{thmmain1}, it is possible to construct both reductions in Theorem~\ref{thmmain2} to be Borel-Lipschitz on small distances.

Finally, we prove an analogue of Theorem~\ref{thmmain2} in the setting of Banach spaces, see Theorem~\ref{thmmain4}. The proof is more or less based on known methods, but it requires some additional work, and we postpone it to an appendix.

\section{Preliminaries} \label{section:prelim}

Let $ \varrho_{H}(A, B) $ denote the Hausdorff distance of two non-empty subsets $ A $ and $ B $ of a metric space. The \emph{Gromov-Hausdorff distance} of non-empty metric spaces $ M $ and $ N $ is defined by
$$ \varrho_{GH} (M, N) = \inf_{\substack{\textnormal{$ Z $ metric space}\\ i_{M} : M \hookrightarrow Z\\ i_{N}: N \hookrightarrow Z}} \varrho_{H} \big( i_{M}(M), i_{N}(N) \big), $$
where the symbol $ \hookrightarrow $ denotes an isometric embedding. Actually, in this work, both the Hausdorff distance $ \varrho_{H} $ and the Gromov-Hausdorff distance $ \varrho_{GH} $ are considered mainly between elements of the space $ F(\mathbb{U}) \setminus \{ \emptyset \} $ defined below, and both distances are regarded as pseudometrics on this space.

A binary relation $ \mathcal{R} \subseteq X \times Y $ is called a \emph{correspondence} between $ X $ and $ Y $ if $ \forall x \in X \exists y \in Y : x \mathcal{R} y $ and $ \forall y \in Y \exists x \in X : x \mathcal{R} y $. Let $ (M, \delta_{M}) $ and $ (N, \delta_{N}) $ be metric spaces and let $ r > 0 $. It is well known and easy to show that if $ \varrho_{GH} (M, N) < r $, then there exists a correspondence $ \mathcal{R} $ between $ M $ and $ N $ such that
$$ m \mathcal{R} n \; \& \; m' \mathcal{R} n' \quad \Rightarrow \quad | \delta_{M}(m, m') - \delta_{N}(n, n')| < 2r. $$

A \emph{Polish space (topology)} means a separable completely metrizable topological space (topology), and a \emph{Polish metric space} means a separable complete metric space. A \emph{Polish group} is a topological group whose topology is Polish. By an action of a group $ G $ on a set $ X $ we mean a mapping $ (g, x) \in G \times X \mapsto g \cdot x \in X $ satisfying $ 1_{G} \cdot x = x $ and $ (gh) \cdot x = g \cdot (h \cdot x) $ for all $ g, h \in G $ and $ x \in X $.

The \emph{Urysohn space} is defined as the (up to isometry) only Polish metric space $ \mathbb{U} $ with the property that for any finite metric space $ A $ and any isometric embedding $ i : B \to \mathbb{U} $, where $ B \subseteq A $, there exists an isometric embedding $ \widetilde{i} : A \to \mathbb{U} $ extending $ i $. It is well known that the Urysohn space contains an isometric copy of every Polish metric space, in fact, the following result holds.

\begin{theorem}[Kat\v{e}tov \cite{katetov}] \label{thmkatetov}
Let $ X $ be a Polish metric space. Then there exists an isometric embedding $ i : X \to \mathbb{U} $ such that any surjective isometry on $ i(X) $ can be extended to a surjective isometry on $ \mathbb{U} $.
\end{theorem}

For a Polish metric space $ (X, \delta_{X}) $, we define
$$ F(X) = \big\{ F \subseteq X : \textnormal{$ F $ is closed} \big\}, $$
and we equip $ F(X) \setminus \{ \emptyset \} $ with the \emph{Wijsman topology}, defined as the coarsest topology for which the function
$$ F \mapsto \delta_{X}(x, F), $$
is continuous for each $ x \in X $. If we moreover add $ \emptyset $ as an isolated point, then $ F(X) $ is a Polish space whose Borel $ \sigma $-algebra is the \emph{Effros Borel structure} of $ X $, defined as the $ \sigma $-algebra generated by the sets $ \{ F \in F(X) : F \cap U \neq \emptyset \} $, $ U \subseteq X $ open, see e.g.~\cite{beer}. Let us note that if $ X $ is universal (in the sense that it contains an isometric copy of every Polish metric space), then $ F(X) $ is a coding of all Polish metric spaces up to isometry. For this purpose, we will employ the Urysohn space and the coding $ F(\mathbb{U}) $. Although this is not the same coding as the coding considered in the preceding works \cite{cdk1} and \cite{cdk2}, it makes no important difference.

For a Polish metric space $ X $, we denote by $ Iso(X) $ the group of all surjective isometries on $ X $ with the topology of pointwise convergence. This is a Polish group, and by its canonical action on $ F(X) $ we mean the action $ I \cdot F = I(F) $. Let us prove a simple lemma.

\begin{lemma} \label{lemmawijsman}
Let $ F(X) \setminus \{ \emptyset \} $ be equipped with the Wijsman topology. Then the canonical action of $ Iso(X) $ on $ F(X) \setminus \{ \emptyset \} $ is continuous. Moreover, the Hausdorff distance $ \varrho_{H} $ is lower semicontinuous on $ F(X) \setminus \{ \emptyset \} $ and there is a system $ \Phi $ of continuous pseudometrics such that $ \varrho_{H} = \sup_{\varphi \in \Phi} \varphi $.
\end{lemma}

\begin{proof}
Assuming $ I_{n} \to I $ and $ F_{n} \to F $, we show that $ I_{n}(F_{n}) \to I(F) $. Given $ x \in X $, we need to check that $ \delta_{X}(x, I_{n}(F_{n})) \to \delta_{X}(x, I(F)) $. We have
\begin{align*}
|\delta_{X}(x & , I_{n}(F_{n})) - \delta_{X}(x, I(F))| = |\delta_{X}(I_{n}^{-1}(x), F_{n}) - \delta_{X}(I^{-1}(x), F)| \\
& \leq |\delta_{X}(I_{n}^{-1}(x), F_{n}) - \delta_{X}(I^{-1}(x), F_{n})| + |\delta_{X}(I^{-1}(x), F_{n}) - \delta_{X}(I^{-1}(x), F)| \\
& \leq \delta_{X}(I_{n}^{-1}(x), I^{-1}(x)) + |\delta_{X}(I^{-1}(x), F_{n}) - \delta_{X}(I^{-1}(x), F)|,
\end{align*}
and it remains to note that $ \delta_{X}(I_{n}^{-1}(x), I^{-1}(x)) \to 0 $ since $ I_{n} \to I $ and $ \delta_{X}(I^{-1}(x), F_{n}) \to \delta_{X}(I^{-1}(x), F) $ since $ F_{n} \to F $.

Concerning the moreover part, it is sufficient to notice that
$$ \varrho_{H}(A, B) = \sup_{x \in X} |\delta_{X}(x, A) - \delta_{X}(x, B)| $$
for all $ A, B \in F(X) \setminus \{ \emptyset \} $.
\end{proof}

Let us consider the action of the group $ S_{\infty} $ of all permutations of $ \mathbb{N} $ on the space $ [1/2, 1]^{[\mathbb{N}]^{2}} $ given by
$$ (\pi \cdot x)(m, n) = x \big( \pi^{-1}(m), \pi^{-1}(n) \big), \quad \{ m, n \} \in [\mathbb{N}]^{2}, $$
for $ \pi \in S_{\infty}, x \in [1/2, 1]^{[\mathbb{N}]^{2}} $, and let a pseudometric $ \sigma $ on $ [1/2, 1]^{[\mathbb{N}]^{2}} $ be defined by
$$ \sigma(x, y) = \sup_{\{ m, n \} \in [\mathbb{N}]^{2}} \big| x(m, n) - y(m, n) \big| $$
for $ x, y \in [1/2, 1]^{[\mathbb{N}]^{2}} $. Let us remark that the points of $ [1/2, 1]^{[\mathbb{N}]^{2}} $ represent the metrics on $ \mathbb{N} $ with values in $ \{ 0 \} \cup [1/2, 1] $. We will apply the following result from the preceding work \cite{cdk1}, for more details, see \cite[Theorem~11]{cdk1} and the proof of \cite[Theorem~26]{cdk1}.

\begin{theorem}[\cite{cdk1}] \label{thmcdk}
The Gromov-Hausdorff distance $ \varrho_{GH} $ is Borel-u.c.~bireducible with the orbit pseudometric $ \varrho_{S_{\infty}, \sigma} $ on $ [1/2, 1]^{[\mathbb{N}]^{2}} $.
\end{theorem}

Let us recall one more result which will play an important role in our construction.

\begin{theorem}[Melleray \cite{melleray}] \label{thmmelleray}
Let $ X $ be a Polish metric space of diameter at most $ 1 $ and let $ G $ be a closed subgroup of $ Iso(X) $. Then there exists an extension $ Y $ of $ X $ such that
\begin{itemize}
\item $ Y $ is a Polish metric space,
\item any member of $ G $ can be extended in a unique way to a surjective isometry on $ Y $,
\item any surjective isometry on $ Y $ is an extension of a member of $ G $.
\end{itemize}
\end{theorem}

Since this theorem is not stated explicitly in \cite{melleray}, it is necessary to give an explanation. Theorem~\ref{thmmelleray} can be proved in the same way as \cite[Theorem~1.1]{melleray} with the following two differences:
\begin{itemize}
\item The part before \emph{Claim} is to be ignored.
\item We want to preserve the distance, so $ d $ should not be replaced by $ \frac{d}{1+d} $ after the proof of \cite[Lemma~3.2]{melleray}. The purpose of this replacement is to fulfill $ \mathrm{diam}(Z) \leq 1 $. However, it holds for the original distance that $ \mathrm{diam}(Z) \leq 3 $, which is sufficient for the final step of the construction. Therefore, $ d $ can be preserved.
\end{itemize}

\section{The Construction}

The aim of this section is to prove the following theorem from which the results introduced above follow.

\begin{theorem} \label{thmmain3}
Let $ G $ be a Polish group acting continuously on a Polish space $ X $. Let $ d $ be a lower semicontinuous pseudometric on $ X $ with $ d \leq 1 $ such that $ d(x, y) = d(gx, gy) $ for any $ x, y \in X $ and $ g \in G $. Moreover, let there be a system $ \Phi $ of lower semicontinuous pseudometrics on $ X $ such that $ d = \sup_{\varphi \in \Phi} \varphi $ and every $ \varphi \in \Phi $ has the property that the mapping
$$ p \mapsto H_{\varphi}(p) \coloneqq \big\{ (x, u) : x \in X, \varphi(p, x) \leq u \leq 1 \big\} $$
from $ X $ to $ F(X \times [0, 1]) $ is Borel.

Then there exists a Borel mapping $ p \mapsto Y_{p} $ from $ X $ to $ F(\mathbb{U}) \setminus \{ \emptyset \} $ such that
$$ \varrho_{GH}(Y_{p}, Y_{q}) \leq \varrho_{Iso(\mathbb{U}), \varrho_{H}}(Y_{p}, Y_{q}) \leq \varrho_{G, d}(p, q) \leq 2 \varrho_{GH}(Y_{p}, Y_{q}) $$
for all $ p, q \in X $.
\end{theorem}

Before proving this theorem, let us show first that Theorems~\ref{thmmain1} and \ref{thmmain2} are its consequences. It is sufficient to prove Theorem~\ref{thmmain1} in the case $ d \leq 1 $, as one can consider $ \min \{ d, 1 \} $ instead of $ d $. So, we just need to show that any continuous pseudometric $ \varphi $ on $ X $ satisfies the requirement from Theorem~\ref{thmmain3}. For an open $ U \subseteq X \times [0, 1] $, we show that $ \{ p \in X : H_{\varphi}(p) \cap U \neq \emptyset \} $ is Borel (actually open). Let us assume that $ p \in X $ is such that $ H_{\varphi}(p) $ intersects $ U $, and let us pick $ (x, u) \in H_{\varphi}(p) \cap U $. If $ u < 1 $, consider $ v > u $ such that $ (x, v) \in U $, and if $ u = 1 $, consider $ v = 1 $. As $ \varphi $ is continuous, for every $ q $ from a neighborhood of $ p $, we have $ \varphi(q, x) \leq v $, and so $ (x, v) \in H_{\varphi}(q) \cap U $.

Concerning Theorem~\ref{thmmain2}, the reducibility of $ \varrho_{GH} $ to $ \varrho_{Iso(\mathbb{U}), \varrho_{H}} $ follows from Theorem~\ref{thmmain3} and Theorem~\ref{thmcdk}, and the reducibility of $ \varrho_{Iso(\mathbb{U}), \varrho_{H}} $ to $ \varrho_{GH} $ follows from Theorem~\ref{thmmain1} and Lemma~\ref{lemmawijsman}.

Let us now turn to the proof of Theorem~\ref{thmmain3}. We notice first that there is a countable $ \Phi' \subseteq \Phi $ such that still $ d = \sup_{\varphi \in \Phi'} \varphi $. Indeed, the open sets $ \{ (x, y, u) \in X \times X \times [0, 1] : u < \varphi(x, y) \}, \varphi \in \Phi, $ form a covering of $ \{ (x, y, u) \in X \times X \times [0, 1] : u < d(x, y) \} $, and due to the Lindel\"{o}f property, countably many of them form a covering as well. We choose a sequence $ \varphi_{1}, \varphi_{2}, \dots $ in $ \Phi' $ such that every element of $ \Phi' $ appears infinitely many times.

Let $ \gamma $ be a compatible right-invariant metric on $ G $ with $ \gamma \leq 1 $ (let us note that the space $ (G, \gamma) $ needs not to be complete). Let $ \delta_{X} $ be a compatible complete metric on $ X $ with $ \delta_{X} \leq 1 $. Let us consider the maximum metric on $ G \times X $, i.e., the metric $ \delta_{G \times X}((g, x), (h, y)) = \max \{ \gamma(g, h), \delta_{X}(x, y) \} $. It is easy to check that the mapping $ I_{h} : (g, x) \mapsto (gh, x) $ is an isometry on $ G \times X $ for every $ h \in G $. We obtain the following claim based on Melleray's result.

\begin{claim} \label{cl1}
There is an extension $ Z $ of $ G \times X $ such that
\begin{itemize}
\item $ Z $ is a Polish metric space of diameter at most $ 2 $,
\item for any $ h \in G $, the isometry $ I_{h} : (g, x) \mapsto (gh, x) $ can be extended in a unique way from $ G \times X $ to a surjective isometry on $ Z $,
\item any surjective isometry on $ Z $ is an extension of $ I_{h} $ for some $ h \in G $.
\end{itemize}
\end{claim}

\begin{proof}
Let $ \overline{G} $ denote the completion of $ (G, \gamma) $. For every $ h \in G $, let $ I_{h}^{*} $ be the unique surjective isometry on $ \overline{G} \times X $ extending $ I_{h} $. We show that the set $ \{ I_{h}^{*} : h \in G \} $ forms a closed subgroup of $ Iso(\overline{G} \times X) $. This is clearly a subgroup, as $ I_{h'}^{*} \circ I_{h}^{*} = I_{hh'}^{*} $, and by the fact that any Polish subgroup of a Polish group is closed (see e.g. \cite[Exercise~9.6]{kechris}), it is sufficient to show that $ h \mapsto I_{h}^{*} $ is a homeomorphism from $ G $ into $ Iso(\overline{G} \times X) $.

First, let $ h_{n} \to h $. Then $ I_{h_{n}}^{*}(g, x) = (gh_{n}, x) \to (gh, x) = I_{h}^{*}(g, x) $ for every $ (g, x) \in G \times X $. So, $ I_{h_{n}}^{*} $ converges pointwise to $ I_{h}^{*} $ on a dense subset of $ \overline{G} \times X $. As the mappings $ I_{h_{n}}^{*} $ and $ I_{h}^{*} $ are isometries, they are equicontinuous, and the pointwise convergence on the whole $ \overline{G} \times X $ follows. Conversely, let $ I_{h_{n}}^{*} \to I_{h}^{*} $. If we pick $ x \in X $ arbitrarily, then $ (h_{n}, x) = I_{h_{n}}^{*}(1_{G}, x) \to I_{h}^{*}(1_{G}, x) = (h, x) $, which means that $ h_{n} \to h $.

So, $ \{ I_{h}^{*} : h \in G \} $ is a closed subgroup of $ Iso(\overline{G} \times X) $ indeed, and we obtain from Theorem~\ref{thmmelleray} an extension $ Z $ of $ \overline{G} \times X $ satisfying all desired properties with the only possible exception of the diameter requirement. However, such property can be easily arranged by changing the metric appropriately (we can take an increasing concave function $ \xi : [0, \infty) \to [0, 2) $ with $ \xi(t) = t $ for $ t \in [0, 1] $ and consider $ \xi \circ \delta_{Z} $ instead of $ \delta_{Z} $).
\end{proof}

Now, let us consider
$$ Y = (Z \times \mathbb{N}) \cup (Z \times [0, 1] \times \mathbb{N}). $$
We define a compatible complete metric $ m $ on $ Y $ in three steps. Concerning the verification of the triangle inequality, some details will be left to the reader. The first step, perhaps the least obvious one, uses a similar idea as the proofs of \cite[Theorem~39]{cdk2} and \cite[Theorem~41]{cdk2}. The choice of the distances between points in $ Z \times \mathbb{N} $ will play a crucial role in the proof of Claim~\ref{cl6}.

(i) On the subset $ Z \times \mathbb{N} $ of $ Y $, we put
$$ m((z_{1}, k_{1}), (z_{2}, k_{2})) = 100 \cdot |2^{k_{1}} - 2^{k_{2}}| + 2^{\min \{ k_{1}, k_{2} \} } \delta_{Z}(z_{1}, z_{2}). $$
The triangle inequality
$$ m((z_{1}, k_{1}), (z_{3}, k_{3})) \leq m((z_{1}, k_{1}), (z_{2}, k_{2})) + m((z_{2}, k_{2}), (z_{3}, k_{3})) $$
can be easily checked when $ k_{2} \geq \min \{ k_{1}, k_{3} \} $ (it is possible to deal separately with the terms $ 100 \cdot |2^{k_{1}} - 2^{k_{2}}| $ and $ 2^{\min \{ k_{1}, k_{2} \} } \delta_{Z}(z_{1}, z_{2}) $). We can suppose that $ k_{1} \leq k_{3} $, so we deal now with the case $ k_{2} < k_{1} \leq k_{3} $, in which the inequality can be simplified to
$$ 200 \cdot 2^{k_{2}} + 2^{k_{1}} \delta_{Z}(z_{1}, z_{3}) \leq 200 \cdot 2^{k_{1}} + 2^{k_{2}} \delta_{Z}(z_{1}, z_{2}) + 2^{k_{2}} \delta_{Z}(z_{2}, z_{3}). $$
It is sufficient to use $ 200 \cdot 2^{k_{2}} \leq 100 \cdot 2^{k_{1}} $ and $ 2^{k_{1}} \delta_{Z}(z_{1}, z_{3}) \leq 2^{k_{1}} \cdot 2 $.

(ii) Let $ k \in \mathbb{N} $ be fixed. To attach $ Z \times [0, 1] \times \{ k \} $ to $ Z \times \mathbb{N} $, we put
$$ m((z_{1}, u_{1}, k), (z_{2}, u_{2}, k)) = |u_{1} - u_{2}| + 2^{k} \delta_{Z}(z_{1}, z_{2}) $$
and
$$ m((z_{1}, l), (z_{2}, u, k)) = u + 10 \cdot 2^{k} + m((z_{1}, l), (z_{2}, k)). $$
The triangle inequality can be verified easily, since $ (Z \times \mathbb{N}) \cup (Z \times [0, 1] \times \{ k \}) $ with $ m $ can be isometrically embedded into $ Z \times \mathbb{N} \times (\{ - 10 \cdot 2^{k} \} \cup [0, 1]) $ with the metric $ ((z_{1}, k_{1}, u_{1}), (z_{2}, k_{2}, u_{2})) \mapsto |u_{1} - u_{2}| + m((z_{1}, k_{1}), (z_{2}, k_{2})) $.

(iii) It remains to define distances between points in $ Z \times [0, 1] \times \{ k_{1} \} $ and points in $ Z \times [0, 1] \times \{ k_{2} \} $ for distinct $ k_{1} $ and $ k_{2} $. Assuming $ k_{1} \neq k_{2} $, we put
$$ m((z_{1}, u_{1}, k_{1}), (z_{2}, u_{2}, k_{2})) = u_{1} + 10 \cdot 2^{k_{1}} + u_{2} + 10 \cdot 2^{k_{2}} + m((z_{1}, k_{1}), (z_{2}, k_{2})). $$
To check the triangle inequality, one can for instance proceed similarly as in the previous step, considering a suitable subset of $ (Z \times \mathbb{N}) \times \prod_{k \in \mathbb{N}} (\{ - 10 \cdot 2^{k} \} \cup [0, 1]) $ with the sum metric.

Once we have defined the space $ Y $ and its metric $ m $, we can consider for every $ p \in X $ the subspace
$$ W_{p} = (Z \times \mathbb{N}) \cup \big\{ (g, x, u, k) : g \in G, x \in X, k \in \mathbb{N}, \varphi_{k}(gp, x) \leq u \leq 1 \big\}. $$
For every $ p \in X $, let $ Y_{p} $ be the closure of $ W_{p} $ in $ Y $. We now provide a series of claims concerning the spaces $ Y_{p} $.

\begin{claim} \label{cl2}
The mapping $ p \mapsto Y_{p} $ from $ X $ to $ F(Y) $ is Borel.
\end{claim}

\begin{proof}
Given an open $ U \subseteq Y $, we need to check that the set $ \{ p \in X : Y_{p} \cap U \neq \emptyset \} $ is Borel. Let us show first that for each $ k \in \mathbb{N} $, the set
$$ V_{k} = \big\{ (g^{-1} x, u) : g \in G, x \in X, u \in [0, 1], (g, x, u, k) \in U \big\} $$
is open. Let us pick $ (g^{-1} x, u) \in V_{k} $ (where $ (g, x, u, k) \in U $). Let $ O \subseteq X $ be an open neighborhood of $ x $ and $ J \subseteq [0, 1] $ be a (relatively) open neighborhood of $ u $ such that $ \{ g \} \times O \times J \times \{ k \} \subseteq U $. Then $ (g^{-1} \cdot O) \times J $ is an open neighborhood of $ (g^{-1} x, u) $ contained in $ V_{k} $.

Now, for $ p \in X $, we have $ Y_{p} \cap U \neq \emptyset $ if and only if $ W_{p} \cap U \neq \emptyset $, and this is equivalent to
$$ (Z \times \mathbb{N}) \cap U \neq \emptyset \quad \textnormal{or} \quad \exists k : H_{\varphi_{k}}(p) \cap V_{k} \neq \emptyset. $$
By the assumption on the pseudometrics in $ \Phi $, the set of all $ p \in X $ with this property is Borel.
\end{proof}

\begin{claim} \label{cl3}
For all $ p \in X $ and $ h \in G $, there is $ I \in Iso(Y) $ that maps $ Y_{p} $ onto $ Y_{hp} $.
\end{claim}

\begin{proof}
If we denote by $ I_{h}^{**} $ the unique isometry on $ Z $ extending $ I_{h} $, then the mapping
$$ (z, k) \mapsto (I_{h}^{**}(z), k), \quad (z, u, k) \mapsto (I_{h}^{**}(z), u, k), \quad \textnormal{(in part. $ (g, x, u, k) \mapsto (gh, x, u, k) $)}, $$
is an isometry on $ Y $ which maps $ W_{hp} $ onto $ W_{p} $. It follows that it maps $ Y_{hp} $ onto $ Y_{p} $ as well.
\end{proof}

\begin{claim} \label{cl4}
For all $ p, q \in X $,
$$ \varrho_{Iso(Y), \varrho_{H}}(Y_{p}, Y_{q}) \leq \varrho_{G, d}(p, q). $$
\end{claim}

\begin{proof}
Let us note that the Hausdorff distance between $ Y_{p} $ and $ Y_{q} $ is at most $ d(p, q) $, which follows from the fact that $ |\varphi_{k}(gp, x) - \varphi_{k}(gq, x)| \leq \varphi_{k}(gp, gq) \leq d(gp, gq) = d(p, q) $ for every $ g \in G, x \in X $ and $ k \in \mathbb{N} $. So, using Claim~\ref{cl3}, we obtain for each $ h \in G $ that $ \varrho_{Iso(Y), \varrho_{H}}(Y_{p}, Y_{q}) \leq \varrho_{H}(Y_{hp}, Y_{q}) \leq d(hp, q) $, and it is sufficient to consider the infimum over $ h \in G $.
\end{proof}

\begin{claim} \label{cl5}
Let $ p \in X, w \in W_{p} $, and let $ A(w, p) = \{ m(w', w) : w' \in W_{p} \} $. Let further $ B \subseteq [0, \infty) $ be such that $ \varrho_{H}(B, A(w, p)) \leq 1 $ and $ \beta_{l} $ be defined by $ \beta_{l} = \inf (B \cap (80 \cdot 2^{l}, \infty)) $ for $ l \in \mathbb{N} $.

If $ w \in Z \times \{ k \} $, then $ \liminf_{l \to \infty} (100 \cdot 2^{l} - \beta_{l}) \in [100 \cdot 2^{k} - 1, 100 \cdot 2^{k} + 1] $.

If $ w \in Z \times [0, 1] \times \{ k \} $, then $ \liminf_{l \to \infty} (100 \cdot 2^{l} - \beta_{l}) \in [90 \cdot 2^{k} - 2, 90 \cdot 2^{k} + 1] $.

Consequently, if $ \varrho_{H}(A(w, p), A(w', q)) \leq 1 $ for some $ p, q \in X, w \in W_{p}, w' \in W_{q} $, then $ w \in Z \times \{ k \} $ is equivalent to $ w' \in Z \times \{ k \} $ and $ w \in Z \times [0, 1] \times \{ k \} $ is equivalent to $ w' \in Z \times [0, 1] \times \{ k \} $.
\end{claim}

\begin{proof}
Let us first consider $ w = (z, k) $. For every $ l \geq k + 3 $,
\begin{itemize}
\item if $ w' \in (Z \times \{ 1, \dots, l-1 \}) \cup (Z \times [0, 1] \times \{ 1, \dots, l-1 \}) $, then $ m(w', w) \leq 1 + 110 \cdot 2^{l-1} + 2^{k} \cdot 2 \leq 80 \cdot 2^{l} - 2 $,
\item if $ w' \in (Z \times \{ l, l+1, \dots \}) \cup (Z \times [0, 1] \times \{ l, l+1, \dots \}) $, then $ m(w', w) \geq 100 \cdot (2^{l} - 2^{k}) $,
\item $ w' = (z, l) $ belongs to $ W_{p} $ and $ m(w', w) = 100 \cdot (2^{l} - 2^{k}) $.
\end{itemize}
So, the set $ A(w, p) $ does not intersect the interval $ (80 \cdot 2^{l} - 2, 100 \cdot (2^{l} - 2^{k})) $ and contains $ 100 \cdot (2^{l} - 2^{k}) $. It follows that $ \beta_{l} \in [100 \cdot (2^{l} - 2^{k}) - 1, 100 \cdot (2^{l} - 2^{k}) + 1] $, and so $ 100 \cdot 2^{l} - \beta_{l} \in [100 \cdot 2^{k} - 1, 100 \cdot 2^{k} + 1] $.

Now, let us consider $ w = (z, u, k) $. For every $ l \geq k + 3 $,
\begin{itemize}
\item if $ w' \in (Z \times \{ 1, \dots, l-1 \}) \cup (Z \times [0, 1] \times \{ 1, \dots, l-1 \}) $, then $ m(w', w) \leq 1 + 110 \cdot 2^{l-1} + 2^{k} \cdot 2 + 1 + 10 \cdot 2^{k} \leq 80 \cdot 2^{l} - 2 $,
\item if $ w' \in (Z \times \{ l, l+1, \dots \}) \cup (Z \times [0, 1] \times \{ l, l+1, \dots \}) $, then $ m(w', w) \geq 100 \cdot (2^{l} - 2^{k}) + u + 10 \cdot 2^{k} $,
\item $ w' = (z, l) $ belongs to $ W_{p} $ and $ m(w', w) = 100 \cdot (2^{l} - 2^{k}) + u + 10 \cdot 2^{k} $.
\end{itemize}
So, the set $ A(w, p) $ does not intersect the interval $ (80 \cdot 2^{l} - 2, 100 \cdot (2^{l} - 2^{k}) + 10 \cdot 2^{k}) $ and intersects the interval $ [100 \cdot (2^{l} - 2^{k}) + 10 \cdot 2^{k}, 100 \cdot (2^{l} - 2^{k}) + 1 + 10 \cdot 2^{k}] $. It follows that $ \beta_{l} \in [100 \cdot (2^{l} - 2^{k}) + 10 \cdot 2^{k} - 1, 100 \cdot (2^{l} - 2^{k}) + 1 + 10 \cdot 2^{k} + 1] $, and so $ 100 \cdot 2^{l} - \beta_{l} \in [90 \cdot 2^{k} - 2, 90 \cdot 2^{k} + 1] $.

Concerning the consequence part, we can choose the same set $ B $ for both points $ w, w' $ (the choice $ B = A(w, p) $ works, as well as the choice $ B = A(w', q) $). It is then sufficient to note that $ B $ determines to which one of the sets $ Z \times \{ k \} $ or $ Z \times [0, 1] \times \{ k \} $ the points $ w, w' $ belong, as the intervals $ [100 \cdot 2^{k} - 1, 100 \cdot 2^{k} + 1] $ together with the intervals $ [90 \cdot 2^{k} - 2, 90 \cdot 2^{k} + 1] $ are pairwise disjoint.
\end{proof}

\begin{claim} \label{cl6}
For all $ p, q \in X $,
$$ \varrho_{G, d}(p, q) \leq 2 \varrho_{GH}(Y_{p}, Y_{q}). $$
\end{claim}

\begin{proof}
Given $ r > \varrho_{GH}(Y_{p}, Y_{q}) $, we need to show that $ 2r \geq \varrho_{G, d}(p, q) $. Since $ r > \varrho_{GH}(Y_{p}, Y_{q}) = \varrho_{GH} (W_{p}, W_{q}) $, there is a correspondence $ \mathcal{R} $ between $ W_{p} $ and $ W_{q} $ such that
$$ a \mathcal{R} a' \; \& \; b \mathcal{R} b' \quad \Rightarrow \quad |m(a, b) - m(a', b')| \leq 2r. $$
As $ \varrho_{G, d}(p, q) \leq 1 $, we can assume that $ 2r < 1 $. In such a case, due to Claim~\ref{cl5}, elements of $ Z \times \{ k \} $ may correspond only to elements of $ Z \times \{ k \} $ itself, and the same holds for $ Z \times [0, 1] \times \{ k \} $. Consequently, the relations
$$ \mathcal{R}_{k} = \big\{ (z, z') \in Z \times Z : (z, k) \mathcal{R} (z', k) \big\}, \quad k \in \mathbb{N}, $$
are correspondences on $ Z $.

It is easy to check that
$$ z \mathcal{R}_{k_{1}} z_{1} \; \& \; z \mathcal{R}_{k_{2}} z_{2} \quad \Rightarrow \quad \delta_{Z}(z_{1}, z_{2}) \leq 2^{-\min \{ k_{1}, k_{2} \} } \cdot 2r, $$
as $ (z, k_{1}) \mathcal{R} (z_{1}, k_{1}), (z, k_{2}) \mathcal{R} (z_{2}, k_{2}) $ implies $ 2r \geq |m((z, k_{1}), (z, k_{2})) - m((z_{1}, k_{1}), (z_{2}, k_{2}))| = | 100 \cdot |2^{k_{1}} - 2^{k_{2}}| + 2^{\min \{ k_{1}, k_{2} \} } \delta_{Z}(z, z) - 100 \cdot |2^{k_{1}} - 2^{k_{2}}| - 2^{\min \{ k_{1}, k_{2} \} } \delta_{Z}(z_{1}, z_{2}) | = 2^{\min \{ k_{1}, k_{2} \} } \delta_{Z}(z_{1}, z_{2}) $.

It follows that for every $ z \in Z $, there is $ I(z) \in Z $ such that
$$ z \mathcal{R}_{k} z' \quad \Rightarrow \quad \delta_{Z}(I(z), z') \leq 2^{-k} \cdot 2r $$
(if we consider $ S_{j} = \bigcup_{k=j}^{\infty} z \mathcal{R}_{k} $, then the diameter of $ S_{j} $ is at most $ 2^{-j} \cdot 2r $, and we can take $ \{ I(z) \} = \bigcap_{j=1}^{\infty} \overline{S_{j}} $). Similarly, for every $ z \in Z $, there is $ J(z) \in Z $ such that
$$ z' \mathcal{R}_{k} z \quad \Rightarrow \quad \delta_{Z}(z', J(z)) \leq 2^{-k} \cdot 2r. $$

We claim that $ I $ and $ J $ are isometries. It is sufficient to consider $ I $ only. Let $ z_{1}, z_{2} \in Z $, and let $ k \in \mathbb{N} $ be arbitrary. If we choose $ z'_{1}, z'_{2} $ such that $ z_{1} \mathcal{R}_{k} z'_{1}, z_{2} \mathcal{R}_{k} z'_{2} $, then $ \delta_{Z}(I(z_{1}), z'_{1}) \leq 2^{-k} \cdot 2r, \delta_{Z}(I(z_{2}), z'_{2}) \leq 2^{-k} \cdot 2r $. At the same time, from $ (z_{1}, k) \mathcal{R} (z'_{1}, k), (z_{2}, k) \mathcal{R} (z'_{2}, k) $ we get $ | m((z_{1}, k), (z_{2}, k)) - m((z'_{1}, k), (z'_{2}, k)) | \leq 2r $, i.e., $ 2^{k} \cdot | \delta_{Z}(z_{1}, z_{2}) - \delta_{Z}(z'_{1}, z'_{2}) | \leq 2r $. By the triangle inequality, $ | \delta_{Z}(z_{1}, z_{2}) - \delta_{Z}(I(z_{1}), I(z_{2})) | \leq 3 \cdot 2^{-k} \cdot 2r $. Since $ k \in \mathbb{N} $ was arbitrary, we finally see that $ \delta_{Z}(I(z_{1}), I(z_{2})) = \delta_{Z}(z_{1}, z_{2}) $.

In order to show that $ I $ is surjective, we further claim that $ J $ is inverse to $ I $. Let $ z \in Z $, and let $ k \in \mathbb{N} $ be arbitrary. If we choose $ z' $ such that $ z \mathcal{R}_{k} z' $, then $ \delta_{Z}(I(z), z') \leq 2^{-k} \cdot 2r $ and $ \delta_{Z}(z, J(z')) \leq 2^{-k} \cdot 2r $. As $ \delta_{Z}(J(I(z)), J(z')) = \delta_{Z}(I(z), z') $, the triangle inequality provides $ \delta_{Z}(J(I(z)), z) \leq 2 \cdot 2^{-k} \cdot 2r $. Since $ k \in \mathbb{N} $ was arbitrary, we see that $ \delta_{Z}(J(I(z)), z) = 0 $.

So, we have shown that $ I $ is a surjective isometry on $ Z $, and it follows that $ I $ is the extension of $ I_{h} $ for some $ h \in G $.

Now, let $ k \in \mathbb{N} $ be arbitrary. If we choose $ z $ such that $ (1_{G}, p) \mathcal{R}_{k} z $ and $ g, x, u $ such that $ (1_{G}, p, 0, k) \mathcal{R} (g, x, u, k) $, then
$$ \big| m((1_{G}, p, k), (1_{G}, p, 0, k)) - m((z, k), (g, x, u, k)) \big| \leq 2r. $$
The involved distances are $ 10 \cdot 2^{k} $ and $ u + 10 \cdot 2^{k} + 2^{k} \delta_{Z}(z, (g, x)) $, hence we get
$$ u + 2^{k} \delta_{Z}(z, (g, x)) \leq 2r. $$
As $ (g, x, u, k) \in W_{q} $, we have $ \varphi_{k}(gq, x) \leq u \leq 1 $, and thus
$$ \varphi_{k}(gq, x) \leq 2r. $$
Since $ \delta_{Z}((h, p), z) = \delta_{Z}(I(1_{G}, p), z) \leq 2^{-k} \cdot 2r $, the triangle inequality provides $ \delta_{Z}((h, p), (g, x)) \leq 2^{-k} \cdot 2r + 2^{-k} \cdot 2r $, i.e.,
$$ \gamma(h, g) \leq 2^{-k} \cdot 4r, \quad \delta_{X}(p, x) \leq 2^{-k} \cdot 4r. $$

So, we have seen that for every $ k \in \mathbb{N} $, there are $ g_{k} \in G $ and $ x_{k} \in X $ such that
$$ \varphi_{k}(g_{k}q, x_{k}) \leq 2r, \quad \gamma(h, g_{k}) \leq 2^{-k} \cdot 4r, \quad \delta_{X}(p, x_{k}) \leq 2^{-k} \cdot 4r. $$
We have $ g_{k} \to h $, $ x_{k} \to p $, and for every $ \varphi \in \Phi' $, we get $ \varphi(hq, p) \leq 2r $, as it is lower semicontinuous and $ \varphi = \varphi_{k} $ infinitely many times. It follows that $ \varrho_{G, d}(q, p) \leq d(hq, p) \leq 2r $.
\end{proof}

To finish the proof of Theorem~\ref{thmmain3}, we apply Theorem~\ref{thmkatetov}. We may suppose that $ Y $ is actually a subspace of $ \mathbb{U} $ such that any surjective isometry on $ Y $ can be extended to a surjective isometry on $ \mathbb{U} $. It is easy to check that
$$ \varrho_{GH}(Y_{p}, Y_{q}) \leq \varrho_{Iso(\mathbb{U}), \varrho_{H}}(Y_{p}, Y_{q}) \leq \varrho_{Iso(Y), \varrho_{H}}(Y_{p}, Y_{q}) $$
for all $ p, q \in X $. Together with Claims \ref{cl4}, \ref{cl6} and \ref{cl2}, this shows that the mapping $ p \mapsto Y_{p} $ works, which completes the proof of Theorem~\ref{thmmain3}.

\begin{remark}
Let us note here that for all $ p \in X $ and $ h \in G $, there is $ I \in Iso(\mathbb{U}) $ that maps $ Y_{p} $ onto $ Y_{hp} $ (this follows from Claim~\ref{cl3}). Hence, our reduction has the property that
$$ p E^{X}_{G} q \quad \Rightarrow \quad Y_{p} E^{F(\mathbb{U})}_{Iso(\mathbb{U})} Y_{q} \quad \Rightarrow \quad \textnormal{$ Y_{p} $ and $ Y_{q} $ are isometric} \quad \Rightarrow \quad \varrho_{GH}(Y_{p}, Y_{q}) = 0. $$
Moreover, if $ d $ is the discrete metric, then this holds with equivalences. In this way, we obtain the result of Clemens, Gao and Kechris that $ E^{X}_{G} $ is Borel reducible to the relation of isometry on $ F(\mathbb{U}) $, as well as to the relation $ E^{F(\mathbb{U})}_{Iso(\mathbb{U})} $, see \cite{clemens, gaokech}.
\end{remark}

\section{Comments and questions} \label{section:comments}

An immediate question concerning Theorem~\ref{thmmain1} remains open.

\begin{question} \label{question1}
Does Theorem~\ref{thmmain1} hold without the requirement of the existence of the system $ \Phi $ of continuous pseudometrics?
\end{question}

We will show in the next remark that the existence of such a system is not guaranteed for a lower semicontinuous pseudometric $ d $. However, we do not know if the more general system from Theorem~\ref{thmmain3} needs to exist. We just know that the mapping $ p \mapsto H_{\varphi}(p) $ needs not to be Borel for a lower semicontinuous pseudometric $ \varphi $, see Remark~\ref{remarknonmeas}. Later, we will show that Question~\ref{question1} has a positive answer in the case of a countable group action, see Proposition~\ref{propcountgrp}.

Let us recall that an equivalence relation $ E $ on a Polish space $ X $ is called \emph{smooth} if there exists a Borel reduction of $ E $ into the relation of equality on a Polish space $ Y $. It is known that every closed (actually every $ G_{\delta} $) equivalence relation is smooth, see \cite[Corollary~1.2]{haloke}. It turns out that the question if a system $ \Phi $ from Theorem~\ref{thmmain1} exists for any lower semicontinuous pseudometric is connected with the question if the reduction witnessing smoothness of a general closed equivalence relation can be found not only Borel, but even continuous. We show in the next remark that the question has a negative answer. Although this is not a surprising finding, we were not able to find a counterexample in the literature.

\begin{remark}
There is a closed equivalence relation $ E $ on a Polish space $ X $ and two points $ a, b \in X $ which are not $ E $-equivalent but have the property that any continuous mapping $ f : X \to Y $ into a Polish space $ Y $ with $ x E y \Rightarrow f(x) = f(y) $ satisfies $ f(a) = f(b) $. It follows that
\begin{itemize}
\item there is no continuous reduction to the relation of equality on a Polish space $ Y $ (although a Borel reduction necessarily exists),
\item the pseudometric $ \varrho_{E} $ given by $ \varrho_{E}(x, y) = 0 $ if $ x E y $ and $ \varrho_{E}(x, y) = 1 $ otherwise does not satisfy the assumption of Theorem~\ref{thmmain1}, as any continuous pseudometric $ \varphi $ with $ \varphi \leq \varrho_{E} $ satisfies $ \varphi(a, b) = 0 < 1 = \varrho_{E}(a, b) $.
\end{itemize}

Let us provide such an example. Let $ A $ be the subset of the Banach space $ \ell_{1} $ consisting of the point $ x = 0 $ and the points $ x_{n, m} = \frac{1}{m} e_{n}, n, m \in \mathbb{N} $. Let $ B $ be the set of real numbers consisting of the point $ p = 0 $ and the points $ p_{k} = \frac{1}{k}, k \in \mathbb{N} $. Further, let $ X $ be the space consisting of two copies of $ A $ and infinitely many copies of $ B $. More precisely,
$$ X = \big( \{ x^{(0)} \} \cup \{ x^{(0)}_{n, m} : n, m \in \mathbb{N} \} \big) \cup \big( \{ x^{(1)} \} \cup \{ x^{(1)}_{n, m} : n, m \in \mathbb{N} \} \big) \cup \bigcup_{m=1}^{\infty} \big( \{ p^{(m)} \} \cup \{ p^{(m)}_{k} : k \in \mathbb{N} \} \big). $$
The distance of two points from different copies of $ A $ or $ B $ can be defined as $ 10 $, say. Finally, we consider the equivalence relation $ E $ with the equivalence classes
$$ \{ x^{(0)} \}, \quad \{ x^{(1)} \}, \quad \{ p^{(m)} \}, m \in \mathbb{N}, \quad \{ x^{(0)}_{n, m}, p^{(m)}_{2n-1} \}, n, m \in \mathbb{N}, \quad \{ x^{(1)}_{n, m}, p^{(m)}_{2n} \}, n, m \in \mathbb{N}. $$

Let us check that $ E $ is closed. Suppose that $ a_{i} \to a $ and $ b_{i} \to b $ in $ X $ and that $ a_{i} E b_{i} $ for every $ i \in \mathbb{N} $. We need to check that $ a E b $. We may suppose that each $ a_{i} $ (each $ b_{i} $) belongs to the same copy of $ A $ or $ B $ as $ a $ (as $ b $). In the case that $ a $ and $ b $ belong to the same copy of $ A $ or $ B $, then $ a_{i} = b_{i} $ for $ i \in \mathbb{N} $, and so $ a = b $. In the opposite case, there are four possibilities, but the argument is the same in each of them. If, say, there is $ m $ such for every $ i $, we have $ a_{i} = x^{(0)}_{n(i), m} $ and $ b_{i} = p^{(m)}_{2n(i)-1} $ for some $ n(i) $, then there is $ n $ such that $ n(i) = n $ for all but finitely many $ i $'s (as the distance of $ x^{(0)}_{n_{1}, m} $ and $ x^{(0)}_{n_{2}, m} $ is always $ 2/m $ for $ n_{1} \neq n_{2} $). Hence, $ a = x^{(0)}_{n, m} $, $ b = p^{(m)}_{2n-1} $, and $ a E b $.

Now, we show that the choice $ a = x^{(0)}, b = x^{(1)} $ works. Let $ f : X \to Y $ be a continuous mapping into a Polish space $ Y $ such that $ x E y \Rightarrow f(x) = f(y) $. We show first that $ f(p^{(m)}) \to f(x^{(0)}) $. Let $ U $ be a neighborhood of $ f(x^{(0)}) $. There is $ m_{0} $ such that $ f(x) \in U $ whenever the distance of $ x $ to $ x^{(0)} $ is less than $ 1/m_{0} $. For all $ m > m_{0} $ and $ n \in \mathbb{N} $, we have $ f(p^{(m)}_{2n-1}) = f(x^{(0)}_{n, m}) \in U $. It follows that $ f(p^{(m)}) \in \overline{U} $ for all $ m > m_{0} $. As $ U $ was an arbitrary neighborhood of $ f(x^{(0)}) $, we finally arrive at $ f(p^{(m)}) \to f(x^{(0)}) $. By the same argument, we can obtain $ f(p^{(m)}) \to f(x^{(1)}) $. Therefore, $ f(x^{(0)}) = f(x^{(1)}) $.
\end{remark}

\begin{remark} \label{remarknonmeas}
Concerning the assumptions of Theorem~\refeq{thmmain3}, we can ask if the mapping $ p \mapsto H_{\varphi}(p) $ is Borel for a general lower semicontinuous pseudometric $ \varphi $. If we restrict our attention on pseudometrics attaining only values $ 0 $ and $ 1 $, we actually ask if the mapping $ x \mapsto [x]_{E} $ from $ X $ to $ F(X) $ is Borel for a general closed equivalence relation $ E $ on a Polish space $ X $. This question has a negative answer. We can consider any closed equivalence without a Borel selector, or we can consider the following example due to No\'e de Rancourt.

Let $ A $ be an analytic non-Borel set in a Polish space $ Y $. Then $ A $ is the projection of a closed subset $ B $ of $ Y \times \mathbb{N}^{\mathbb{N}} $. Let us add an isolated point $ p $ to $ \mathbb{N}^{\mathbb{N}} $ and consider the space $ X = B \cup (Y \times \{ p \}) $ and the equivalence $ (y_{1}, \nu_{1}) E (y_{2}, \nu_{2}) \Leftrightarrow y_{1} = y_{2} $. Then $ B $ is an open subset of $ X $ but the set $ \{ y \in Y : [(y, p)]_{E} \cap B \neq \emptyset \} = A $ is not Borel.
\end{remark}

It is known and easy to show that an equivalence relation $ E $ on a Polish space $ X $ is smooth if and only if there exist Borel subsets $ A_{n} $ of $ X $ such that $ xEy $ iff $ \{ n : x \in A_{n} \} = \{ n : y \in A_{n} \} $. For this reason, the following lemma can be seen as a continuous version of the above mentioned result from \cite{haloke} that every closed equivalence is smooth. Its proof is inspired by the proof of \cite[Lemma~5.1]{haloke} and a proof of Urysohn's lemma.

\begin{lemma} \label{lemmalscpseud}
Let $ d $ be a lower semicontinuous pseudometric on a Polish space $ X $. Then there are Borel functions $ s_{n} : X \to \mathbb{R}, n \in \mathbb{N}, $ such that
$$ d(p, q) = \sup_{n \in \mathbb{N}} |s_{n}(p) - s_{n}(q)|, \quad p, q \in X. $$
\end{lemma}

\begin{proof}
For $ A \subseteq X $ and $ r > 0 $, let us denote
$$ (A)_{r} = \{ x \in X : d(x, A) < r \}, $$
where $ d(x, A) $ means the distance of $ x $ to $ A $ in the pseudometric space $ (X, d) $. It holds that
$$ \textrm{$ A $ is analytic} \quad \Rightarrow \quad \textrm{$ (A)_{r} $ is analytic}, $$
since $ (A)_{r} $ is a projection of the intersection of $ X \times A $ with the $ F_{\sigma} $ set $ \{ (x, y) \in X \times X : d(x, y) < r \} $.

Let us pick $ \eta > 0 $ and $ p, q \in X $ with $ d(p, q) > \eta $. We can choose open neighborhoods $ U \ni p $ and $ V \ni q $ such that $ d(x, y) > \eta $ for $ x \in U $ and $ y \in V $. By a recursive procedure, we will find for every rational number $ u \in [0, 1] $ a Borel set $ B_{u} $ in the way that
\begin{itemize}
\item $ B_{0} = U $, $ B_{1} = X \setminus V $,
\item $ B_{u} \subseteq B_{v} $ for $ u < v $,
\item $ d(B_{u}, X \setminus B_{v}) \geq \eta \cdot (v - u) $ for $ u < v $,
\end{itemize}
where $ d(B_{u}, X \setminus B_{v}) $ means the distance of these two sets in the pseudometric space $ (X, d) $. Note first that $ d(B_{0}, X \setminus B_{1}) = d(U, V) \geq \eta = \eta \cdot (1 - 0) $. Let $ u_{1} = 0, u_{2} = 1 $, and let the rational numbers in $ (0, 1) $ be enumerated by $ u_{3}, u_{4}, \dots $. Suppose that $ n \geq 2 $ and that $ B_{u_{1}}, \dots, B_{u_{n}} $ are already chosen. Let us consider	$$ C = \mathop{\bigcup_{1 \leq k \leq n}}_{u_{k} < u_{n+1}} (B_{u_{k}})_{\eta \cdot (u_{n+1}-u_{k})}, \quad D = \mathop{\bigcup_{1 \leq k \leq n}}_{u_{k} > u_{n+1}} (X \setminus B_{u_{k}})_{\eta \cdot (u_{k}-u_{n+1})}. $$
As $ C $ and $ D $ are disjoint analytic sets, using the Lusin separation principle (see e.g. \cite[Theorem~14.7]{kechris}), we can find a Borel set $ B_{u_{n+1}} $ which separates $ C $ from $ D $, that is, $ C \subseteq B_{u_{n+1}} \subseteq X \setminus D $. It is easy to check that the procedure works.

Once the sets $ B_{u} $ are found, we define
$$ s(x) = \eta \cdot \inf \big( \{ 1 \} \cup \{ u \in [0, 1] \cap \mathbb{Q} : x \in B_{u} \} \big), \quad x \in X. $$
It is easy to show that $ s $ is a Borel function and that $ s(y) - s(x) = \eta - 0 = \eta $ for $ x \in U $ and $ y \in V $. Let us check that $ |s(x) - s(y)| \leq d(x, y) $ for all $ x, y \in X $. We can suppose that $ s(x) > s(y) $. Let $ u, v $ be rational with $ s(y)/\eta < u < v < s(x)/\eta $. Then $ y \in B_{u} $ and $ x \notin B_{v} $, and so $ d(x, y) \geq d(B_{u}, X \setminus B_{v}) \geq \eta \cdot (v - u) $. Since $ u $ can be arbitrarily close to $ s(y)/\eta $ and $ v $ can be arbitrarily close to $ s(x)/\eta $, we obtain $ d(x, y) \geq s(x) - s(y) $.

Now, let $ \eta > 0 $ be fixed. The construction above provides a covering of the set $ \{ (p, q) \in X \times X : d(p, q) > \eta \} $ by open sets of the form $ U \times V $, for each of which there is a Borel function $ s : X \to \mathbb{R} $ such that $ |s(x) - s(y)| \leq d(x, y) $ for all $ x, y \in X $ and $ s(y) - s(x) = \eta $ for $ x \in U $ and $ y \in V $. Due to the Lindel\"{o}f property, there is a countable system $ S_{\eta} $ of such functions satisfying $ d(p, q) > \eta \Rightarrow \sup_{s \in S_{\eta}} |s(p) - s(q)| \geq \eta $.

Finally, the system $ S = \bigcup \{ S_{\eta} : \eta > 0, \eta \in \mathbb{Q} \} $ has the desired property $ d(p, q) = \sup_{s \in S} |s(p) - s(q)| $ for all $ p, q \in X $.
\end{proof}

In the next proposition, we give a positive answer to Question~\ref{question1} in the case that the group $ G $ is countable.

\begin{proposition} \label{propcountgrp}
Let $ G $ be a countable group acting on a Polish space $ X $, and let the action be Borel. Let $ d $ be a lower semicontinuous pseudometric on $ X $ such that $ d(x, y) = d(gx, gy) $ for any $ x, y \in X $ and $ g \in G $. Then $ \varrho_{G, d} $ is Borel-u.c.~reducible to the Gromov-Hausdorff distance $ \varrho_{GH} $.
\end{proposition}

\begin{proof}
Let $ \mathbb{R}^{\mathbb{N} \times G} $ be equipped with the distance $ \sigma(a, b) = \sup_{(n, g) \in \mathbb{N} \times G} |a(n, g) - b(n, g)| $ and with the action $ (h \cdot a)(n, g) = a(n, gh) $. By Theorem~\refeq{thmmain1}, the pseudometric $ \varrho_{G, \sigma} $ on $ \mathbb{R}^{\mathbb{N} \times G} $ is Borel-u.c.~reducible to $ \varrho_{GH} $, so it is sufficient to show that $ \varrho_{G, d} $ is Borel-u.c.~reducible to $ \varrho_{G, \sigma} $. In fact, there is a Borel-isometric reduction.
	
Indeed, if $ s_{n}, n \in \mathbb{N}, $ are as in Lemma~\ref{lemmalscpseud} and $ f : X \to \mathbb{R}^{\mathbb{N} \times G} $ is defined by $ f(x)(n, g) = s_{n}(gx) $, then we have
\begin{align*}
\varrho_{G, \sigma}(f(x) & , f(y)) = \inf_{h} \sigma(h \cdot f(x), f(y)) = \inf_{h} \sup_{n, g} \big| (h \cdot f(x))(n, g) - f(y)(n, g) \big| \\
 & = \inf_{h} \sup_{n, g} \big| f(x)(n, gh) - f(y)(n, g) \big| = \inf_{h} \sup_{n, g} \big| s_{n}(ghx) - s_{n}(gy) \big| \\
 & = \inf_{h} \sup_{g} d(ghx, gy) = \inf_{h} \sup_{g} d(hx, y) = \inf_{h} d(hx, y) = \varrho_{G, d}(x, y)
\end{align*}
for all $ x, y \in X $.
\end{proof}

Analogously as in Corollary~\ref{corzerodist}, we obtain in the setting of Proposition~\ref{propcountgrp} that the equivalence relation given by $ x E^{X}_{G, d} y \Leftrightarrow \varrho_{G, d}(x, y) = 0 $ is reducible to $ E_{GH} $, and so $ E_{1} $ is not reducible to it. Let us note that $ E^{X}_{G, d} $ is automatically Borel in such setting, so it possess much lower complexity than $ E_{GH} $. Despite of that, we still do not know if $ E^{X}_{G, d} $ is reducible to an orbit equivalence relation, even for the following simple example.

\begin{question} \label{question2}
Is the equivalence relation
$$ E^{[0, 1]^{\mathbb{Z}}}_{\mathbb{Z}, \sigma} = \big\{ (x, y) \in [0, 1]^{\mathbb{Z}} : \big( \forall \varepsilon > 0 \, \exists k \in \mathbb{Z} \, \forall l \in \mathbb{Z} : |x(l+k) - y(l)| < \varepsilon \big) \big\} $$
Borel reducible to an orbit equivalence relation?
\end{question}

It is conjectured that for any Borel equivalence relation $ E $, either $ E $ is reducible to an orbit equivalence relation, or $ E_{1} $ is reducible to $ E $, see \cite[Conjecture~1]{hjokech}. Thus, a negative answer to Question~\ref{question2} would be an unexpected and remarkable result.

\appendix

\section{Banach spaces setting}

The purpose of this appendix is to prove a version of Theorem~\ref{thmmain2} in the framework of Banach spaces. We need to recall some definitions first, starting with an analogue of the Urysohn space.

A separable Banach space $ \mathbb{G} $ is called a \emph{Gurariy space} if, for every $ \varepsilon > 0 $, every finite-dimensional Banach spaces $ X $ and $ Y $ with $ X \subseteq Y $ and every linear isometry $ f : X \to \mathbb{G} $, there exists an extension $ g : Y \to \mathbb{G} $ of $ f $ such that $ (1 + \varepsilon)^{-1} \Vert y \Vert \leq \Vert g(y) \Vert \leq (1 + \varepsilon) \Vert y \Vert $ for every $ y \in Y $. It is known that there exists only one Gurariy space up to isometry (\cite{lusky}, see also \cite{kubsol}). We will need the following analogue of Theorem~\ref{thmkatetov} which follows from \cite[Lemma~3.8]{benyaacov}.

\begin{theorem}[Ben Yaacov \cite{benyaacov}] \label{thmbenyaacov}
Let $ X $ be a separable Banach space. Then there exists a linear isometric embedding $ i : X \to \mathbb{G} $ such that any surjective linear isometry on $ i(X) $ can be extended to a surjective linear isometry on $ \mathbb{G} $.
\end{theorem}

We define
$$ \mathcal{SE}(\mathbb{G}) = \{ X \in F(\mathbb{G}) : \textnormal{$ X $ is linear} \}. $$
This is a closed subset of $ F(\mathbb{G}) $ (with the Wijsman topology), since it consists of elements $ X $ of $ F(\mathbb{G}) $ such that $ \mathrm{dist} (x+y, X) \leq \mathrm{dist} (x, X) + \mathrm{dist} (y, X) $ and $ \mathrm{dist} (\lambda x, X) = |\lambda| \mathrm{dist} (x, X) $ for $ x, y \in \mathbb{G} $ and $ \lambda \in \mathbb{R} $. Further, by $ Iso_{L}(\mathbb{G}) $ we denote the subgroup of $ Iso(\mathbb{G}) $ of all surjective linear isometries on $ \mathbb{G} $.

For two linear subspaces $ E, F $ of a normed linear space $ Z $, we define their modified Hausdorff distance by
$$ \tilde{\varrho}_{H}(E, F) = \varrho_{H}(B_{E}, B_{F}), $$
where $ B_{E} $ and $ B_{F} $ denote the closed unit balls of $ E $ and $ F $. The \emph{Kadets distance} of normed linear spaces $ X $ and $ Y $ is defined by
$$ \varrho_{K} (X, Y) = \inf_{\substack{\textnormal{$ Z $ normed linear space}\\ i_{X} : X \hookrightarrow Z\\ i_{Y}: Y \hookrightarrow Z}} \tilde{\varrho}_{H} \big( i_{X}(X), i_{Y}(Y) \big), $$
where the symbol $ \hookrightarrow $ denotes a linear isometric embedding.

Now, we can formulate the main result of this section. Of course, the distances $ \tilde{\varrho}_{H} $ and $ \varrho_{K} $ are regarded as pseudometrics on $ \mathcal{SE}(\mathbb{G}) $ here.

\begin{theorem} \label{thmmain4}
The Kadets distance $ \varrho_{K} $ is Borel-u.c.~bireducible with the orbit pseudometric $ \varrho_{Iso_{L}(\mathbb{G}), \tilde{\varrho}_{H}} $ on $ \mathcal{SE}(\mathbb{G}) $.
\end{theorem}

Let us notice first that the pseudometric $ \varrho_{Iso_{L}(\mathbb{G}), \tilde{\varrho}_{H}} $ is reducible to $ \varrho_{GH} $. The action of $ Iso_{L}(\mathbb{G}) $ on $ \mathcal{SE}(\mathbb{G}) $ is continuous by Lemma~\ref{lemmawijsman}, and the pseudometric
$$ \tilde{\tilde{\varrho}}_{H}(E, F) = \sup_{x \in B_{\mathbb{G}}} \big| \mathrm{dist}(x, E) - \mathrm{dist}(x, F) \big|, \quad E, F \in \mathcal{SE}(\mathbb{G}), $$
satisfies the requirement from Theorem~\refeq{thmmain1}, hence the orbit pseudometric $ \varrho_{Iso_{L}(\mathbb{G}), \tilde{\tilde{\varrho}}_{H}} $ is reducible to $ \varrho_{GH} $. It remains to note that $ \tilde{\tilde{\varrho}}_{H} \leq 2 \tilde{\varrho}_{H} $ and $ \tilde{\varrho}_{H} \leq 2 \tilde{\tilde{\varrho}}_{H} $.

Further, we know from Theorem~\ref{thmcdk} that the pseudometrics $ \varrho_{GH} $ and $ \varrho_{S_{\infty}, \sigma} $ on $ [1/2, 1]^{[\mathbb{N}]^{2}} $ are bireducible. It is also known that both of them are bireducible with the Kadets distance $ \varrho_{K} $ (see \cite[Theorems~46 and~48]{cdk2}).

So, concerning the proof of Theorem~\ref{thmmain4}, it is sufficient to find a reduction of $ \varrho_{S_{\infty}, \sigma} $ to $ \varrho_{Iso_{L}(\mathbb{G}), \tilde{\varrho}_{H}} $. To find such a reduction, we use a construction from \cite[Section~4]{cdk2}. Let us fix numbers $ \alpha $ and $ \delta $ such that
$$ 1 < \alpha < \alpha + \delta \leq \frac{200}{199}. $$
For every $ f : [\mathbb{N}]^{2} \to [1/2, 1] $, we define an equivalent norm $ \Vert \cdot \Vert_{f} $ on $ \ell_{2} $ by
$$ \Vert x \Vert_{f} = \sup \Big( \{ \Vert x \Vert_{\ell_{2}} \} \cup \Big\{ \frac{1}{\sqrt{2}} \cdot \big( \alpha + \delta \cdot f(m, n) \big) \cdot |x_{m}+x_{n}| : \{ m, n \} \in [\mathbb{N}]^{2} \Big\} \Big) $$
for $ x = \{ x_{k} \}_{k=1}^{\infty} \in \ell_{2} $. It is easy to show that this is an equivalent norm indeed, as $ \Vert x \Vert_{\ell_{2}} \leq \Vert x \Vert_{f} \leq \frac{200}{199} \Vert x \Vert_{\ell_{2}} $ (see \cite{cdk2}).

The following result follows from \cite[Lemma~51]{cdk2} (for an explanation, see the proof of \cite[Theorem~48]{cdk2}, step (5)).

\begin{lemma}[\cite{cdk2}] \label{lemmacdk}
There is a constant $ C > 0 $ such that
$$ \varrho_{S_{\infty}, \sigma}(f, g) \leq C \varrho_{K} \big( (\ell_{2}, \Vert \cdot \Vert_{f}), (\ell_{2}, \Vert \cdot \Vert_{g}) \big), \quad f, g \in [1/2, 1]^{[\mathbb{N}]^{2}}. $$
\end{lemma}

The next step uses a similar idea as the proof of $ (2) \Rightarrow (1) $ in \cite[Proposition~3.15]{cddk}. We define a subset of $ c_{00}([1/2, 1]^{[\mathbb{N}]^{2}} \times \mathbb{N}) $ by
\begin{align*}
\Omega = \mathrm{co} \bigg( & \bigcup_{f} \Big\{ \sum_{k} x_{k} e_{f, k} : x \in c_{00}(\mathbb{N}), \Vert x \Vert_{f} \leq 1 \Big\} \\
 & \cup \bigcup_{f, g} \Big\{ \sum_{k} x_{k} e_{f, k} - \sum_{k} x_{k} e_{g, k} : x \in c_{00}(\mathbb{N}), \delta \cdot \sigma(f, g) \cdot \Vert x \Vert_{\ell_{2}} \leq 1 \Big\} \\
 & \cup \bigcup_{f, g, \{ m, n \} } \Big\{ c \cdot \big( e_{f, m} + e_{f, n} - e_{g, m} - e_{g, n} \big) : |c| \cdot \delta \cdot |f(m, n) - g(m, n)| \cdot \sqrt{2} \leq 1 \Big\} \bigg),
\end{align*}
where $ \{ e_{f, k} \}_{f, k} $ stands for the canonical basis of $ c_{00}([1/2, 1]^{[\mathbb{N}]^{2}} \times \mathbb{N}) $. Let us denote the corresponding Minkowski functional by $ \gamma $. Before we prove several properties of $ \Omega $ and $ \gamma $, let us note that it follows immediately from the definition that
\begin{itemize}
\item $ \gamma(\sum_{k} x_{k} e_{f, k}) \leq \Vert x \Vert_{f} $,
\item $ \gamma(\sum_{k} x_{k} e_{f, k} - \sum_{k} x_{k} e_{g, k}) \leq \delta \cdot \sigma(f, g) \cdot \Vert x \Vert_{\ell_{2}} $,
\item $ \gamma(e_{f, m} + e_{f, n} - e_{g, m} - e_{g, n}) \leq \delta \cdot |f(m, n) - g(m, n)| \cdot \sqrt{2} $.
\end{itemize}

\begin{claim} \label{clapp1}
We have $ \gamma(\sum_{k} x_{k} e_{f, k}) = \Vert x \Vert_{f} $ for $ f \in [1/2, 1]^{[\mathbb{N}]^{2}} $ and $ x = \{ x_{k} \}_{k} \in c_{00}(\mathbb{N}) $.
\end{claim}

\begin{proof}
For notational purposes, we change the notation $ f $ and $ x $ in the claim to $ h $ and $ y $. We already know that $ \gamma(\sum_{k} y_{k} e_{h, k}) \leq \Vert y \Vert_{h} $. The verification of the opposite inequality consists of two tasks. Let us note first that any linear functional $ u^{*} $ on $ c_{00}([1/2, 1]^{[\mathbb{N}]^{2}} \times \mathbb{N}) $ for which $ u^{*} \leq 1 $ on $ \Omega $ satisfies $ u^{*} \leq \gamma $.
	
(i) We check that $ \gamma(\sum_{k} y_{k} e_{h, k}) \geq \Vert y \Vert_{\ell_{2}} $. We can suppose that $ \Vert y \Vert_{\ell_{2}} = 1 $. Let us consider $ u^{*} \in c_{00}([1/2, 1]^{[\mathbb{N}]^{2}} \times \mathbb{N})^{\#} $ given by $ u^{*}(e_{f, k}) = y_{k} $. Then $ u^{*} \leq 1 $ on $ \Omega $, since
\begin{itemize}
\item $ u^{*}(\sum_{k} x_{k} e_{f, k}) = \sum_{k} x_{k} u^{*}(e_{f, k}) = \sum_{k} x_{k} y_{k} \leq \Vert x \Vert_{\ell_{2}} \Vert y \Vert_{\ell_{2}} = \Vert x \Vert_{\ell_{2}} \leq \Vert x \Vert_{f} \leq 1 $ whenever $ \Vert x \Vert_{f} \leq 1 $,
\item $ u^{*}(\sum_{k} x_{k} e_{f, k} - \sum_{k} x_{k} e_{g, k}) = \sum_{k} x_{k} y_{k} - \sum_{k} x_{k} y_{k} = 0 \leq 1 $,
\item $ u^{*}(c \cdot (e_{f, m} + e_{f, n} - e_{g, m} - e_{g, n})) = c \cdot (y_{m} + y_{n} - y_{m} - y_{n}) = 0 \leq 1 $.
\end{itemize}
It follows that $ \gamma(\sum_{k} y_{k} e_{h, k}) \geq u^{*}(\sum_{k} y_{k} e_{h, k}) = \sum_{k} y_{k} u^{*}(e_{h, k}) = \sum_{k} y_{k}^{2} = \Vert y \Vert_{\ell_{2}}^{2} = \Vert y \Vert_{\ell_{2}} $.
	
(ii) We check that $ \gamma(\sum_{k} y_{k} e_{h, k}) \geq \frac{1}{\sqrt{2}} \cdot (\alpha + \delta \cdot h(m, n)) \cdot |y_{m}+y_{n}| $ for any $ \{ m, n \} \in [\mathbb{N}]^{2} $. We can suppose that $ y_{m} + y_{n} \geq 0 $. Let us consider $ u^{*} \in c_{00}([1/2, 1]^{[\mathbb{N}]^{2}} \times \mathbb{N})^{\#} $ given by $ u^{*}(e_{f, m}) = u^{*}(e_{f, n}) = \frac{1}{\sqrt{2}} \cdot (\alpha + \delta \cdot f(m, n)) $ and $ u^{*}(e_{f, k}) = 0 $ for $ k \notin \{ m, n \} $. Then $ u^{*} \leq 1 $ on $ \Omega $, since the points in the formula for $ \Omega $ satisfy
\begin{itemize}
\item $ u^{*}(\sum_{k} x_{k} e_{f, k}) = \sum_{k} x_{k} u^{*}(e_{f, k}) = \frac{1}{\sqrt{2}} \cdot (\alpha + \delta \cdot f(m, n)) (x_{m} + x_{n}) \leq \Vert x \Vert_{f} \leq 1 $,
\item $ u^{*}(\sum_{k} x_{k} e_{f, k} - \sum_{k} x_{k} e_{g, k}) = \frac{1}{\sqrt{2}} \cdot (\alpha + \delta \cdot f(m, n)) (x_{m} + x_{n}) - \frac{1}{\sqrt{2}} \cdot (\alpha + \delta \cdot g(m, n)) (x_{m} + x_{n}) = \frac{1}{\sqrt{2}} \cdot \delta \cdot (f(m, n) - g(m, n)) (x_{m} + x_{n}) \leq \delta \cdot \sigma(f, g) \cdot \Vert x \Vert_{\ell_{2}} \leq 1 $,
\item $ u^{*}(c \cdot (e_{f, m'} + e_{f, n'} - e_{g, m'} - e_{g, n'})) = c \cdot \frac{1}{\sqrt{2}} \cdot \delta \cdot (f(m, n) - g(m, n)) | \{ m', n' \} \cap \{ m, n \} | \leq |c| \cdot \delta \cdot |f(m, n) - g(m, n)| \cdot \sqrt{2} \leq 1 $.
\end{itemize}
It follows that $ \gamma(\sum_{k} y_{k} e_{h, k}) \geq u^{*}(\sum_{k} y_{k} e_{h, k}) = \sum_{k} y_{k} u^{*}(e_{h, k}) = \frac{1}{\sqrt{2}} \cdot (\alpha + \delta \cdot h(m, n)) \cdot (y_{m} + y_{n}) = \frac{1}{\sqrt{2}} \cdot (\alpha + \delta \cdot h(m, n)) \cdot |y_{m}+y_{n}| $.
\end{proof}

\begin{claim} \label{clapp2}
For $ \pi \in S_{\infty} $, the linear operator $ T_{\pi} $ given by $ e_{f, k} \mapsto e_{\pi \cdot f, \pi(k)} $ satisfies $ T_{\pi}(\Omega) = \Omega $, and so it is an isometry on $ (c_{00}([1/2, 1]^{[\mathbb{N}]^{2}} \times \mathbb{N}), \gamma) $.
\end{claim}

\begin{proof}
It is straightforward to check that $ T_{\pi} $ maps the set of the points $ \sum_{k} x_{k} e_{f, k}, f \in [1/2, 1]^{[\mathbb{N}]^{2}}, x \in c_{00}(\mathbb{N}), \Vert x \Vert_{f} \leq 1, $ onto itself. The same holds for the other two collections of points in the definition of $ \Omega $.
\end{proof}

\begin{claim} \label{clapp3}
Let $ f, g \in [1/2, 1]^{[\mathbb{N}]^{2}} $ and let $ u \in \mathrm{span} \{ e_{f, k} : k \in \mathbb{N} \} $ satisfy $ \gamma(u) \leq 1 $. Then there is $ v \in \mathrm{span} \{ e_{g, k} : k \in \mathbb{N} \} $ satisfying $ \gamma(v) \leq 1 $ such that $ \gamma(v - u) \leq 2 \delta \cdot \sigma(f, g) $.
\end{claim}

\begin{proof}
Let $ u $ be expressed by $ u = \sum_{k} x_{k} e_{f, k} $, where $ x = \{ x_{k} \}_{k} \in c_{00}(\mathbb{N}) $, and let $ u' = \sum_{k} x_{k} e_{g, k} $. We know that $ \gamma(u - u') = \gamma(\sum_{k} x_{k} e_{f, k} - \sum_{k} x_{k} e_{g, k}) \leq \delta \cdot \sigma(f, g) \cdot \Vert x \Vert_{\ell_{2}} $. By Claim~\ref{clapp1}, we have $ \Vert x \Vert_{\ell_{2}} \leq \Vert x \Vert_{f} = \gamma(u) \leq 1 $, and so
$$ \gamma(u - u') \leq \delta \cdot \sigma(f, g), $$
in particular $ \gamma(u') \leq 1 + \delta \cdot \sigma(f, g) $. Choosing $ v = \frac{1}{1 + \delta \cdot \sigma(f, g)} u' $, we arrive at $ \gamma(v) \leq 1 $ and $ \gamma(v - u) \leq \gamma(u' - v) + \gamma(u - u') = \gamma(\delta \cdot \sigma(f, g) \cdot v) + \gamma(u - u') \leq 2 \delta \cdot \sigma(f, g) $.
\end{proof}

\begin{claim} \label{clapp4}
The mapping $ f \mapsto e_{f, k} $ is continuous for every $ k $ (in the sense that $ \gamma(e_{g, k} - e_{f, k}) \to 0 $ as $ g \to f $ in $ [1/2, 1]^{[\mathbb{N}]^{2}} $).
\end{claim}

\begin{proof}
For all $ \{ m, n \} \in [\mathbb{N}]^{2} $, the mapping $ f \mapsto e_{f, m} + e_{f, n} $ is continuous, as $ \gamma(e_{f, m} + e_{f, n} - e_{g, m} - e_{g, n}) \leq \delta \cdot |f(m, n) - g(m, n)| \cdot \sqrt{2} $ for each $ f $ and $ g $. It remains to note that $ e_{f, k} = \frac{1}{2} [(e_{f, k} + e_{f, k+1}) + (e_{f, k} + e_{f, k+2}) - (e_{f, k+1} + e_{f, k+2})] $ for each $ f $.
\end{proof}

Now, let $ W $ be the completion of the quotient $ c_{00}([1/2, 1]^{[\mathbb{N}]^{2}} \times \mathbb{N}) / M $, where $ M = \{ u \in c_{00}([1/2, 1]^{[\mathbb{N}]^{2}} \times \mathbb{N}) : \gamma(u) = 0 \} $. Let $ Q : c_{00}([1/2, 1]^{[\mathbb{N}]^{2}} \times \mathbb{N}) \to W $ denote the quotient mapping. Define
$$ W_{f} = \overline{\mathrm{span}} \{ Qe_{f, k} : k \in \mathbb{N} \}, \quad f \in [1/2, 1]^{[\mathbb{N}]^{2}}. $$
We see that
\begin{itemize}
\item $ W_{f} $ is isometric to $ (\ell_{2}, \Vert \cdot \Vert_{f}) $ (Claim~\ref{clapp1}),
\item for $ f \in [1/2, 1]^{[\mathbb{N}]^{2}} $ and $ \pi \in S_{\infty} $, there is a surjective linear isometry on $ W $ which maps $ W_{f} $ onto $  W_{\pi \cdot f} $ (Claim~\ref{clapp2}),
\item $ \tilde{\varrho}_{H}(W_{f}, W_{g}) \leq 2 \delta \cdot \sigma(f, g) $ (Claim~\ref{clapp3}),
\item the mapping $ f \mapsto Qe_{f, k} $ is continuous for every $ k \in \mathbb{N} $ (Claim~\ref{clapp4}).
\end{itemize}

From the last property, we obtain that $ W $ is separable. By Theorem~\ref{thmbenyaacov}, we can suppose that $ W $ is a subspace of $ \mathbb{G} $ and that any surjective linear isometry on $ W $ can be extended to a surjective linear isometry on $ \mathbb{G} $. So, for $ f \in [1/2, 1]^{[\mathbb{N}]^{2}} $ and $ \pi \in S_{\infty} $, there is a surjective linear isometry on $ \mathbb{G} $ which maps $ W_{f} $ onto $  W_{\pi \cdot f} $. It follows that for $ f, g \in [1/2, 1]^{[\mathbb{N}]^{2}} $, we have
$$ \varrho_{Iso_{L}(\mathbb{G}), \tilde{\varrho}_{H}}(W_{f}, W_{g}) \leq 2 \delta \cdot \varrho_{S_{\infty}, \sigma}(f, g). $$
By Lemma~\ref{lemmacdk}, 
$$ \varrho_{S_{\infty}, \sigma}(f, g) \leq C \varrho_{K}(W_{f}, W_{g}), $$
and it is clear that
$$ \varrho_{K}(W_{f}, W_{g}) \leq \varrho_{Iso_{L}(\mathbb{G}), \tilde{\varrho}_{H}}(W_{f}, W_{g}). $$
To show that $ f \mapsto W_{f} $ is the desired reduction, it remains to note that it is a Borel mapping, which is a consequence of the continuity of the mappings $ f \mapsto Qe_{f, k} $.

\end{document}